\documentclass{article}
\usepackage{latexsym}
\usepackage{epsfig}
\usepackage{amsmath}
\usepackage{amssymb}
\usepackage{amsthm}
\newtheorem{theorem}{Theorem}[section]

\newtheorem{lemma}[theorem]{Lemma}

\theoremstyle{remark}
\newtheorem{remark}[theorem]{Remark}

\theoremstyle{definition}

\newcommand{\angles}[2]{\langle #1,#2 \rangle}

\DeclareMathOperator\VecM{vec}

\usepackage{algorithm}
\usepackage{algpseudocode}

\begin{document}

\title{Robustifying networks for flow problems against edge failure}

\author{
  Artyom Klyuchikov \\
  % \texttt{klyuchikov.artem@huawei.com}
  \and
  Roland Hildebrand \\
  MIPT\footnote{Moscow Institute of Physics and Technology, Moscow, Russia}\\
  \texttt{khildebrand.r@mipt.ru} \\
  \and
  Sergei Protasov \\
  MIPT
  %\texttt{khildebrand.r@mipt.ru}
  \and
  Alexander Rogozin \\
  MIPT\\
  \texttt{aleksandr.rogozin@phystech.edu}
  \and
  Alexei Chernov\\
  MIPT\\
  \texttt{chernov.av@mipt.ru}
}

\maketitle

\begin{abstract}
We consider the robust version of a multi-commodity network flow problem. The robustness is defined with respect to the deletion, or failure, of edges. While the flow problem itself is a polynomially-sized linear program, its robust version is a saddle-point problem with discrete variables. We present two approaches for the solution of the robust network flow problem. One way is to formulate the problem as a bigger linear program. The other is to solve a multi-level optimization problem, where the linear programs appearing at the lower level can be solved by the dual simplex method with a warm start. We then consider the problem of robustifying the network. This is accomplished by optimally using a fixed budget for strengthening certain edges, i.e., increasing their capacity. This problem is solved by a sequence of linear programs at the upper level, while at the lower levels the mentioned dual simplex algorithm is employed.
\end{abstract}

\section{Introduction}

Networks and flow problems on them appear in many practically relevant situations. These can be server networks used for data traffic or distributed computations, telecommunication and transportation networks, or power grid networks for the distribution of electrical energy. The pure multi-commodity flow problem (MCFP) arising in such contexts are expressible as linear programs (LP) and hence comparably easy to solve. More information about the MCFP, its applications and solution methods can be found in, e.g., \cite{Salimifard20},\cite{Averbuch05},\cite{Ouorou00} and the references therein.

However, the solution of the linear program is often not satisfactory since additional constraints can be present which complicate the situation and, in particular, render the problem non-convex. In this work we address how to cope with a particular kind of such constraints, related to robustness of the solution. Robust counterparts of convex programs, in particular, linear programs, have been formulated in \cite{NemirovskiRobust}. The robustness is understood against uncertainties in the continuous data of the problem. The solution of the robust counterpart is essentially a saddle point, providing the best value to the objective function under the worst realization of the uncertainty. The robust counterpart of an LP with ellipsoidal uncertainty is a second-order cone program and can also easily be solved with methods of conic programming.

In the case of flow problems on networks we are confronted with a different robustness issue, which is more discrete in nature and hence more difficult to handle. In many practical situations the network edges are prone to deletion by failure. The source of uncertainty considered here is hence the failure of edges of the network. Following the paradigm in \cite{NemirovskiRobust}, we shall measure the robustness of the network by the maximal performance degradation if an arbitrary subset of $k$ edges is deleted.

It is, however, not possible to design a flow which remains feasible under arbitrary edge deletions, since any edge to which a positive flow is assigned can potentially be deleted. The robust counterpart of the flow problem in the sense of \cite{NemirovskiRobust} hence does not exist. 

Instead, we shall seek to \textit{robustify} the network by allocating a constrained budget of additional capacity to the critical edges. This is accomplished by minimizing the mentioned maximal performance degradation with respect to the distribution of the budget. This ultimately leads to a three-level min-max-min optimization problem. Here the two minimization problems are with respect to continuous variables, namely the flows for the inner and the additional capacities for the outer optimization level, and the maximum is with respect to discrete variables modeling the locations of edge failures.

While the inner minimization problem is the well-known LP solving the MCFP, the middle maximization presents a challenge, as it is formally a mixed-binary linear program. We consider two ways to cope with this difficulty. 

The first way is feasible if the number of failed edges is very small, such that all configurations with failed edges can be explicitly incorporated into the problem formulation. This leads to a large linear program. The second approach consists of employing the dual simplex method for solving the flow problems with $k$ failed edges, warm started with the optimal solutions of those problems with $k-1$ failed edges. This way the dual simplex method is used in the branch-and-bound algorithms for solving mixed-integer linear programs \cite{Wolsey98}. 

The remainder of the paper is structured as follows. In Section \ref{sec:model} we describe the model of the network and formulate the throughput and latency criteria. In Section \ref{sec:warm_start} we briefly explain how to use the dual simplex method to efficiently solve daughter LPs. In Section \ref{sec:robust} we formulate the problem of optimal robustification of the network with respect to the robustness measure. In Section \ref{sec:conclusions} we provide some conclusions.

\section{Nominal problem formulation} \label{sec:model}

\subsection{Network model} 

The network is modeled by a directed graph $G = (V,E)$ with $n$ vertices and $m$ edges $e = (i,j) \in E$ being equipped by capacities $b_e > 0$ which limit the throughput. The goal is to satisfy the demand for transportation of an entity through the network. This demand is modeled by a demand matrix $D$, with element $d_{ij}$ specifying how much of the entity has to be taken from vertex $i$ to vertex $j$.

The flow problem will be formulated edge-wise. We introduce variables $f_e^s$ that denote the amount of flow originating from source $s$ and traversing the oriented edge $e$. The flows are stored in a matrix $F = [f_e^s]_{e,s=1,1}^{m,n}\in\mathbb R^{m\times n}$.

There are three kinds of constraints imposed on the flows.

\medskip

\noindent\textit{Capacity constraints}: the sum of the flows over one edge should not exceed its capacity. Hence for each $e = 1, \ldots, m$ we have
\begin{equation} \label{capacity_constraint}
    \sum_{s=1}^n f_e^s\leq b_e,
\end{equation}
which can be written in matrix form as $F{\bf 1}\leq b$.

\medskip

\noindent\textit{Demand constraints}: Each source generates its own flow which must satisfy balance equations at every vertex. In addition, the demand must be satisfied at the target vertices. Let us fix the source $s$ and consider some node $i$. Consider the flow originating from $s$ to target vertex $i$. We have the balance equation
\begin{align*}
	\Delta_i^s = \sum_{e\in E_{in}(i)} f_e^s - \sum_{e\in E_{out}(i)} f_e^s.
\end{align*}
If $i\neq s$, this amount equals the flow running from $s$ to $i$, i.e. $\Delta_i^s = d_{si}$. If $i = s$, this amount is the sum of flows originating from $s$ with inverse sign, i.e. $\Delta_s^s = -\sum_{k\neq s} d_{sk}$. In matrix form the demand constraints can hence be written as $NF = -L_D$. Here $L_D$ is the Laplacian of the demand matrix. Note that the Laplacian is in general not symmetric.

\medskip

\noindent\textit{Non-negativity constraints}: The flows must be non-negative, which amounts to the element-wise inequality $F\geq 0$.

\iffalse

Moreover, we will need the sum of flows going over network. It can be calculated as
\begin{align*}
	f_{sum} = \sum_{s=1}^n \sum_{e\in E_{out}} f_e^s.
\end{align*}
Note that if demand constraints are met, we have $f_{sum} = \sum_{i,j=1}^n d_{ij} = {\bf 1}^\top D{\bf 1}$, i.e. the sum of flows is a constant.

\fi

\subsection{Throughput}

In this section we define the throughput of the network as the optimal value of a maximum flow problem. The net flow from the source to the target nodes has to remain proportional to the demand matrix, but the scale factor is maximized.

More concretely, introduce the \textit{throughput} as the solution of the \textit{maximal concurrent flow problem}
\begin{align*}
	\max_{\lambda, F\geq 0}~ &\lambda\tag{T}\label{eq:throughput} \\
	\text{s.t. } &F{\bf 1}\leq b \\
	&NF = -\lambda L_D
\end{align*}
Here $\lambda$ denotes the proportion of satisfied demands. This problem is a linear program (LP).

\subsection{Relation to load balance}

Let us discuss how throughput is related to load balance (in fact, it is equivalent). Initially, load balance is defined as the maximally loaded edge. Optimal load balance is the solution of the following problem.
\begin{align*}
	\min_{F\geq 0}~ &\max_{e=1,\ldots,m} \frac{[F{\bf 1}]_e}{[b]_e} \\
	\text{s.t. } &NF = -L_D
\end{align*}
It is equivalent to
\begin{align*}
	\min_{\theta, F\geq 0}~ &\theta \tag{LB}\label{eq:load_balance} \\
	\text{s.t. } &F{\bf 1}\leq \theta b \\
	&NF = -L_D
\end{align*}
\begin{lemma}
	Let all bandwidths be positive, i.e. $b > 0$ and let at least one of the demands be nonzero. Then problems \eqref{eq:throughput} and \eqref{eq:load_balance} are equivalent in the sense that given a solution of \eqref{eq:throughput} we can immediately obtain a solution of \eqref{eq:load_balance} and vice versa.
\end{lemma}
\begin{proof}
	Let $(F_\lambda^*, \lambda^*)$ be the solution of \eqref{eq:throughput} and $(F_\theta^*, \theta^*)$ be the solution of \eqref{eq:load_balance}. Since all bandwidths are positive and at least one demand is nonzero, we have $\lambda^* > 0,~ \theta^* > 0$. Let us show that $\lambda^*\theta^* = 1$.
	
	On the one hand, since $(F_\lambda^*, \lambda^*)$ satisfies the constraints in \eqref{eq:throughput}, $(F_\lambda^* / \lambda^*, 1/\lambda^*)$ satisfies the constraints in \eqref{eq:load_balance}. Therefore, $\theta^* \leq 1/\lambda^*$. On the other hand, $(F_\theta^*, \theta^*)$ is feasible for problem \eqref{eq:load_balance}, and therefore $(F_\theta^*/\theta^*, 1/\theta^*)$ satisfies the constraints in \eqref{eq:throughput}. We have that $\lambda^* \geq 1/\theta^*$. Summing up, we conclude that $\theta^*\lambda^* = 1$.
	
	Now let us write KKT conditions for both problems. The dual variables are denoted by $x_\lambda, y_\lambda$ for \eqref{eq:throughput} and by $x_\theta, y_\theta$ for \eqref{eq:load_balance}.
	\begin{align*}
		\text{Problem } &\eqref{eq:throughput} &\text{Problem } \eqref{eq:load_balance} \\
		-x_\lambda^*{\bf 1}^\top+ N^\top y_\lambda^*&\leq 0  &x_\theta^*{\bf 1}^\top+ N^\top y_\theta^*\leq 0 \\
		-\angles{y_\lambda^*, D} + 1 &= 0  &-\angles{x_\theta^*, u} + 1 = 0 \\
		-F_\lambda^*{\bf 1}+ b&\geq 0  &F_\theta^*{\bf 1} - \theta^* b\leq 0 \\
		NF_\lambda^* + \lambda^* L_D &= 0  &NF_\theta^* + L_D = 0
	\end{align*}
	Substituting $F_\lambda^* = \lambda^*F_\theta^*$ to KKT conditions for \eqref{eq:throughput}, we get
	\begin{align*}
		-F_\lambda^*{\bf 1} + b &= -\lambda^* F_\theta^*{\bf 1} + b = -\lambda^*\left(F_\theta^*{\bf 1} - \frac{1}{\lambda^*} b\right) = -\lambda^*(F_\theta^*{\bf 1} - \theta^* b)\leq 0, \\
		NF_\lambda^* + \lambda^* L_D &= \lambda^* NF_\theta^* + \lambda^* L_D = 0,
	\end{align*}
	i.e. $F_\lambda^* = \lambda^*F_\theta^*$ is a solution of \eqref{eq:throughput}. Analogously, it can be shown that $F_\theta^* = F_\lambda^* / \lambda^*$ is a solution of \eqref{eq:load_balance}.
\end{proof}

\subsection{Average latency}

Latency is the measure of the delay of a bit of information. We describe several definitions: linear latency, nonlinear latency and max-average latency.

By latency we understand the sum of delays over all bits of information. Given the solution of \eqref{eq:throughput} $\lambda_{\max}$ and a scalar $\beta\in(0, 1]$, we introduce the \textit{average latency} as the solution of the problem
\begin{align*}
	\min_{F\geq 0}~ &\frac{\angles{c(F{\bf 1}), {\bf 1}}}{\beta\lambda_{\max}{\bf 1}^\top D{\bf 1}} \tag{L}\label{eq:avg_latency}\\
	\text{s.t. } &F{\bf 1}\leq b \\
	&NF = -\beta \lambda_{\max} L_D
\end{align*}
Here $c$ is some element-wise applied function, examples will be given further below. Note that the latency is defined w.r.t. maximal throughput. The more the network is loaded, the larger is the delay which occurs when satisfying all the demand. Here the scalar $\beta$ measures the ratio of network load. 

In this work we mostly concentrate on the value $\beta = 0.9$. The total delay of flows, given by $\langle c(F{\bf 1}), {\bf 1}\rangle$, is divided by the total flow. Provided that the demands are met, the sum of flows equals $\beta\lambda_{\max}{\bf 1}^\top D{\bf 1}$. Since $\beta\lambda_{\max}{\bf 1}^\top D{\bf 1}$ is a constant, problem~\eqref{eq:avg_latency} is an LP.

The vector $c(F{\bf 1})$ contains the edge delays, it depends on the vector of total flows over the individual edges $F{\bf 1}$. In general, the delay at an edge depends only on the flow over this edge, but not on the flow over other edges. We consider several scenarios:
\begin{align*}
	[c(f)]_e &= c_e f_e, \tag{La}\label{eq:linear_latency} \\
	[c(f)]_e &= \frac{c_e f_e}{1 - f_e/b_e}, \tag{Lb}\label{eq:nonlinear_latency_inv} \\
	[c(f)]_e &= c_e(1 - \log(1 - f_e/b_e)). \tag{Lc}\label{eq:nonlinear_latency_log}
\end{align*}
We will further refer to \eqref{eq:linear_latency}, \eqref{eq:nonlinear_latency_inv}, \eqref{eq:nonlinear_latency_log} meaning optimization problem \eqref{eq:avg_latency} with the corresponding objective. Note that \eqref{eq:linear_latency} results in a linear programming problem, while \eqref{eq:nonlinear_latency_inv} and \eqref{eq:nonlinear_latency_log} lead to nonlinear programs.

\begin{remark}
	Despite the simplicity of the introduced concepts some of them pose substantial challenges for their computational implementation. In particular, function \eqref{eq:nonlinear_latency_inv} is computationally unstable. Its value grows quickly as the flow over edge $f_e$ approaches the limiting bandwidth $b_e$. In order to provide stability for the algorithms in practice we multiply \eqref{eq:nonlinear_latency_inv} by a very small positive constant $\alpha_c$. Surely this is not an asymptotic remedy but a provision of stability for some particular cases.
\end{remark}

\section{Dual simplex method with warm start} \label{sec:warm_start}

In this section we briefly describe how the dual simplex method can be used to solve an LP if the solution of a closely related LP is already available. Consider the LP
\[ \min_x\,\langle c,x \rangle:\quad Ax = b,\ Cx \leq d,\ \langle a,x \rangle \leq \beta.
\]
Here we singled out a scalar inequality with right-hand side $\beta \in \mathbb R$. Suppose that the solution $x^*$ of the LP is available together with an optimal basis $B^*$, i.e., a subset of constraints which are inactive. The remaining active constraints, interpreted as equalities, together with the equality constraints of the LP define $x^*$ uniquely as the solution of a system of linear equations.

We then consider a modified daughter LP
\[ \min_x\,\langle c,x \rangle:\quad Ax = b,\ Cx \leq d,\ \langle a,x \rangle \leq \beta - \delta,
\]
where $\delta > 0$. This means that the separate inequality has been strengthened by the amount of $\delta$.

If the separate inequality has been inactive, i.e., part of the basis $B^*$, then the slack value $\beta - \langle a,x^* \rangle$, which is initially nonnegative, diminishes by $\delta$. In case $\beta - \langle a,x^* \rangle \geq \delta$ the point $x^*$ remains optimal, and the daughter LP possesses the same optimal basis $B^*$. If, on the contrary, $\beta - \langle a,x^* \rangle < \delta$, then the slack becomes negative in the daughter LP and the solution $x^*$ becomes infeasible. The basis $B^*$ remains dual feasible, however, and the daughter LP can be solved by the dual simplex method with the basis $B^*$ as initial iterate.

If the separate inequality has been active, i.e., the inequality is not part of $B^*$, then $B^*$ defines a solution which is different from $x^*$ for the daughter LP. This solution may be feasible or not. In case it is feasible, the new solution defined by $B^*$ is also optimal for the daughter LP. In the opposite case it remains dual feasible and $B^*$ can again be used as a warm start for the dual simplex method.

In our case the daughter LP is obtained from the parent LP by deletion of an edge $e$. This modification can be modeled by setting the corresponding edge capacity $b_e$ to zero. The daughter LP hence differs from the parent LP by one strengthened inequality \eqref{capacity_constraint}, which is exactly the situation described above.

\subsection{Computation of derivatives of the objective function} \label{sec:derivatives_of_LP_objective}

Problems \eqref{eq:opt_robust_throughput}, \eqref{eq:optimization_robust_latency} require the optimization of a function whose value can be computed by solving an LP, where the arguments $\delta b$ of the function additively enter the right-hand side of the constraints of the LP. Here the middle layer is a discrete optimization problem, and interfers only at a set of measure zero in the space of arguments $\delta b$.

We hence need a means to compute the derivative of the value of the LP with respect to its parameters entering the right-hand side. Let us show how these derivatives can be extracted from the optimal simplex tableau.

Consider the LP
\[ \min_x\,\langle c,x \rangle: \quad A_{eq}x = b_{eq},\ Ax \leq b.
\]
We are interested in the computation of the derivatives of the value of the LP with respect to the entry $b_i$ of $b$. When bringing the LP into standard form, the inequalities are turned into equalities $Ax + \beta = b$ by the introduction of nonnegative slacks $\beta$. In particular, the scalar inequality where $b_i$ appears is turned into an equality by the slack $\beta_i$.

Suppose we are in possession of the optimal simplex tableau for the LP. This tableau corresponds to some partition of the variables into basic and non-basic ones. We shall assume that in a neighbourhood of the current value of $b$ the optimal partition does not change. This assumption will hold for almost all values of $b$, since the optimal partition is piece-wise constant as a function of $b$.

The variable $\beta_i$ may be basic or nonbasic in the optimal vertex of the LP. If it is basic, then the inequality involving $b_i$ is not active, and a small enough change in $b_i$ will be absorbed a corresponding change in the slack $\beta_i$. This does not change the value of the cost function, however, since it depends only on the variables $x$. Hence the partial derivative of the value of the LP with respect to $b_i$ vanishes.

If the variable $\beta_i$ is non-basic, then the inequality is active, and a change in $b_i$ effectively moves the optimal vertex. As explained in the previous section, a change in $b_i$ amounts to a redefinition of $\beta_i$ by an additive constant which equals minus the change in $b_i$. Equivalently, we may assume that the value of $\beta_i$ changes in the opposite direction with respect to the change in $b_i$, while all other non-basic variables stay equal to zero. But the value of the LP is an affine function of the non-basic variables (as long as the partition into basic and non-basic variables stays the same), and the coefficients of the dependence are given by the cost row of the optimal simplex tableau. Thus the partial derivative of the value with respect to $b_i$ equals minus the coefficient of the non-basic slack $\beta_i$ in the optimal simplex tableau. 

\subsection{Adding constraints} \label{sec:adding_constraints}

In this section we consider a slightly different situation. Suppose we have solved an LP and obtained the optimal simplex tableau. Now we want to solve another LP, which differs from the previous one by an additional inequality constraint.

The optimal table

\medskip

\begin{tabular}{c|c}
	$-c_0$ & $c$ \\ 
	\hline
	$b$ & $A$
\end{tabular}

\medskip

encodes the LP
\[ \min_{x_N,x_B \geq 0}\,(\langle c,x_N \rangle + c_0): \qquad x_B + Ax_N = b,
\]
where $x_N$ is the subvector of non-basic variables, $x_B$ the vector of basic variables, and $(x_N,x_B) = (0,b)$ the optimal vertex.

Now we want to add a constraint of the form
\[ \langle a_N,x_N \rangle + \langle a_B,x_B \rangle \leq b_0,
\]
where $a_N,a_B$ are coefficient vectors of appropriate length, and $b_0 \in \mathbb R$.

We turn the constraint into an equality by introducing a slack $\beta$ and replace $x_B$ by its expression as a function of the non-basic variables. Then the constraint becomes
\[ \langle a_N - A^Ta_B,x_N \rangle + \beta = b_0 - \langle a_B,b \rangle.
\]
Since the number of equalities increased by one, we have to add a variable to the basic set. Here the new equation come in handy, as it expresses the new slack $\beta$ as an affine function of the already existing non-basic variables. We append $\beta$ to the basic set and obtain the new table

\medskip

\begin{tabular}{c|c}
	$-c_0$ & $c$ \\ 
	\hline
	$b$ & $A$ \\
	$b_0 - \langle a_B,b \rangle$ & $a_N - A^Ta_B$
\end{tabular}

\medskip

Note that the cost row $c$ remains unchanged. Since we assumed the original tableau to be optimal, we have $c \geq 0$ and hence the new tableau is dual feasible. 

If in addition $b_0 - \langle a_B,b \rangle \geq 0$, then also the right-hand side vector of the tableau is nonnegative and the tableau is already optimal. This corresponds to the situation when the new constraint is satisfied by the previous optimal point. If $b_0 - \langle a_B,b \rangle < 0$, then the new tableau is merely dual feasible and can be solved by the dual simplex method.

\section{Network robustification problem} \label{sec:robust}

\subsection{Robust throughput}

In this section we define the robust throughput of the network as the worst-case throughput after $q$ edges are deleted. We obtain the optimization problem
\begin{align*}
	\min_{e_1,\ldots,e_q = 1,\ldots,m}~~ \max_{\lambda, F\geq 0}~~ &\lambda \tag{RT}\label{eq:robust_throughput}\\
	\text{s.t. } &F{\bf 1}\leq b - \sum_{i=1}^q b_{e_i}\xi_{e_i} \\
	&NF = -\lambda L_D
\end{align*}
Note that the robust throughput equals zero if the graph edge connectivity between a pair of edges with non-zero demand is less or equal $q$ (i.e. less or equal $q$ edges can be thrown away and the graph becomes disconnected with the source and target nodes ending up in different connection components).

\subsection{Robust latency}

In the same way we may define robust latency as the optimal value of the problem
\begin{align*}
	\max_{e_1,\ldots,e_q = 1,\ldots,m}~~ \min_{F\geq 0}~~ &\frac{\langle c(F{\bf 1}), {\bf 1} \rangle}{\beta\lambda_{\max}{\bf 1}^\top D{\bf 1}} \tag{RL}\label{eq:robust_latency}\\
	\text{s.t. } &F{\bf 1}\leq b - \sum_{i=1}^q b_{e_i}\xi_{e_i} \\
	&NF = -\beta\lambda_{\max} L_D
\end{align*}

\subsection{Throughput robustification}

We ultimately would like to optimize the network w.r.t. the robustness measure \eqref{eq:robust_throughput}, i.e., to robustify the network against failure of at most $q$ edges. Given a budget $B > 0$, we can spend it on increasing bandwidths or adding new edges. The network optimization problem can be formulated as follows:
\begin{center}
	\textit{Optimize the network within the budget $B$ in such a way that the metrics are as good as possible subject to adversarial deletion of at most $q$ edges.}
\end{center}
In other words, we obtain a multi-level optimization problem.

Optimization of the robust throughput can be formulated as
\begin{align*}
	\max_{\delta b\geq 0}~~\min_{e_1,\ldots,e_q = 1,\ldots,m}~~ \max_{\lambda, F\geq 0}~~ &\lambda \tag{ORT}\label{eq:opt_robust_throughput} \\
	\text{s.t. } &F{\bf 1}\leq (b + \delta b) - \sum_{i=1}^q (b_{e_i} + \delta b_{e_i})\xi_{e_i} \\
	&NF = -\lambda L_D \\
	&{\bf 1}^\top\delta b\leq B
\end{align*}

\subsection{Latency robustification}

In this section we formulate the problem of optimization of the robustness measure \eqref{eq:robust_latency}. Similarly to the throughput case we obtain the problem
\begin{align*}
	\min_{\delta b\geq 0}~~\max_{e_1,\ldots,e_q = 1,\ldots,m}~~ \min_{F\geq 0}~~ &\langle c(F{\bf 1}), {\bf 1} \rangle \tag{ORL}\label{eq:optimization_robust_latency} \\
	\text{s.t. } &F{\bf 1}\leq (b + \delta b) - \sum_{i=1}^q (b_{e_i} + \delta b_{e_i})\xi_{e_i} \\
	&NF = -L_D \\
	&{\bf 1}^\top\delta b\leq B
\end{align*}

\section{Solution of the problems}

\subsection{Throughput optimization}

In this section we consider how to solve the problem of throughput optimization. 

Let us investigate how to solve problem \eqref{eq:throughput} by the simplex method. Let us bring the problem into standard form:
\[ \min_{F \geq 0,\beta \geq 0,\lambda \geq 0}\,(-\lambda):\quad NF = -\lambda L_D,\ F{\bf 1} + \beta = b,
\]
where $\beta \in \mathbb R^m$ is a slack to turn the capacity constraint into an equality. For convenience we defined the variable $\lambda$ as being nonnegative, which does not restrict generality. In this form the LP has $nm + m + 1$ variables and $n^2 + m$ equality constraints.

However, not all equality constraints are independent. If the matrix $N$ has a left kernel vector $\zeta \in \mathbb R^n$, then the Laplacian $L_D$ necessarily must have the same kernel vector, otherwise the problem is infeasible. This yields $n$ linear dependencies between the equality constraints, however, which must first be removed from the LP. Note that there always exists the non-zero kernel vector ${\bf 1} \in \mathbb R^n$, so this degeneracy removal step cannot be avoided.

The redundant constraints are best removed by crossing out some rows from the matrix-valued equality $NF = -\lambda L_D$. Denote the resulting submatrices which are obtained from $N,L_D$ by $\tilde N$, $\tilde L_D$, with number of rows $\tilde n$. Then we obtain the non-degenerate LP
\[ \min_{F \geq 0,\beta \geq 0,\lambda \geq 0}\,(-\lambda):\quad \tilde NF = -\lambda \tilde L_D,\ F{\bf 1} + \beta = b
\]
with $nm + m + 1$ variables and $n\tilde n + m$ equality constraints.

Accordingly, the simplex method will identify $n\tilde n + m$ variables as basic and $(n-\tilde n)m + 1$ as nonbasic on this LP. Let $\eta \subset \{1,\dots,m\}$ be an index set of cardinality $\tilde n$ such that the submatrix $\tilde N_{*\eta}$ is regular, and let $\bar\eta$ be its complement. Then we may define the initial basic and non-basic variable vectors by
\[ B = [ \VecM(F_{\eta*}); \beta ],\quad NB = [ \VecM(F_{\bar\eta*}); \lambda ].
\]
Here the vec operator arranges the entries of the argument matrix column by column, and the different objects in the brackets are stacked one upon another. Let us determine the vertex corresponding to this basic set. By definition the non-basic variables are zero, so the equality constraints become
\[ \tilde NF = \tilde N_{*\eta}F_{\eta*} + \tilde N_{*\bar\eta}F_{\bar\eta*} = \tilde N_{*\eta}F_{\eta*} = 0,\ F{\bf 1} + \beta = b.
\]
By regularity of $\tilde N_{*\eta}$ we obtain $F_{\eta*} = 0$, hence $F = 0$ and $\beta = b \geq 0$. Thus the defined basic set is feasible. Starting with it we may solve the LP by the primal simplex method. In order to build the initial feasible simplex tableau, we need to express the basic variables and the cost function as functions of the non-basic variables. This is accomplished by the formulas
\[ F_{\eta*} =  - \tilde N_{*\eta}^{-1}\tilde N_{*\bar\eta} \cdot F_{\bar\eta*} - \tilde N_{*\eta}^{-1}\tilde L_D \cdot \lambda,\ \beta_{\eta} = \tilde N_{*\eta}^{-1}\tilde N_{*\bar\eta} \cdot F_{\bar\eta*}{\bf 1} + \tilde N_{*\eta}^{-1}\tilde L_D{\bf 1} \cdot \lambda + b_{\eta},
\]
\[ \beta_{\bar\eta} = - F_{\bar\eta*}{\bf 1} + b_{\bar\eta},\ -\lambda = (-1) \cdot \lambda.
\]
If we arrange the variables in the basic set and the non-basic set as above, we arrive at the simplex tableau in the Table.

\medskip

\begin{table}
	\centering
	\begin{tabular}{cc|cc}
		& & $\VecM(F_{\bar\eta*})$ & $\lambda$ \\
		& 0 & 0 & -1 \\
		\hline
		$\VecM(F_{\eta*})$ & 0 & $I \otimes (\tilde N_{*\eta}^{-1}\tilde N_{*\bar\eta})$ & $\VecM(\tilde N_{*\eta}^{-1}\tilde L_D)$ \\
		$\beta_{\eta}$ & $b_{\eta}$ & $-{\bf 1}^T \otimes (\tilde N_{*\eta}^{-1}\tilde N_{*\bar\eta})$ & $-\tilde N_{*\eta}^{-1}\tilde L_D{\bf 1}$ \\
		$\beta_{\bar\eta}$ & $b_{\bar\eta}$ & ${\bf 1}^T \otimes I$ & 0
	\end{tabular}
	\caption{Initial simplex tableau for solving the main LP for throughput computation.}
	\label{initial_tableau_throughput}
\end{table}

\medskip

Here the first row and column indicate the order of the basic and non-basic variables, i.e., which rows and columns of the tableau are associated to which variables, the second row and column are the cost row and the right-hand side column, respectively, with minus the current cost value in the corner, and the rest are the coefficients of the linear dependence of the basic variables as functions of the non-basic ones.

Let us summarize the steps more compactly:
\begin{itemize}
	\item determine a maximal index set $R \subset \{1,\dots,n\}$ such that the corresponding rows of $N$ are linearly independent
	\item if the $n - |R|$ linear independent left kernel vectors of $N$ are not kernel vectors of $L_D$, then the LP is infeasible
	\item construct the submatrix $\tilde N = N_{R*}$ of the independent rows
	\item determine a maximal index set $\eta \subset \{1,\dots,m\}$ such that the corresponding columns of $\tilde N$ are linearly independent
	\item build the initial simplex tableau according to Table \ref{initial_tableau_throughput}
	\item solve the LP by the primal simplex method
\end{itemize}

\subsection{Robust throughput computation}

We now pass to the robust throughput, computed in \eqref{eq:robust_throughput}. Note that since in the inner problem we changed the objective by changing its sign and passing to a minimization problem, we have to pass to maximization in the outer problem.
For each set $\{e_1,\dots,e_q\}$ of edges to be removed, the corresponding inner LP differs from the main LP \eqref{eq:throughput} in $q$ entries of the right-hand side of the constraint $F \cdot {\bf 1} \leq b$. Setting an element $b_i$ of $b$ to zero amounts to subtracting the amount $b_i$ from the corresponding slack $\beta_i$. This can be accomplished by solving the LP with the dual simplex method, as described in Section \ref{sec:warm_start}. Here we can remove the edges one by one, each time solving an LP which differs in one entry of the right-hand side from the previous one. If $q$ or $m$ are not too large, we may compute the robust throughput by an exhaustive search on the combinations $\{e_1,\dots,e_q\}$.

\medskip

Finally we discuss problem \eqref{eq:opt_robust_throughput}. Here the outer optimization problem seeks to maximize the robust throughput by increasing the right-hand side vector $b$ within an allotted budget $B$. In our standard form formulation, the outer optimization problem becomes a minimization problem.

The value of modified problem \eqref{eq:robust_throughput} is minimized with respect to the right-hand side vector $b$, which varies over a shifted simplex
\[ \Delta = b^0 + \{ \delta b \geq 0 \mid \langle \delta b,{\bf 1} \rangle \leq B \},
\]
where by $b^0$ we denoted the initial capacity vector.

The value of a LP, in particular problem \eqref{eq:throughput}, is a convex function of the right-hand side $b$. This does not change if we set some entries $\{ b_i \mid i \in I \}$ of this vector in the LP to zero. This makes the value independent of these particular arguments $b_i$, $i \in I$, but the dependency on the other arguments $b_j$, $j \not\in I$ is still convex. Hence the overall dependence on $b$ remains convex. If we thereafter take the maximum of the value over a finite set of problems (the middle optimization problem in \eqref{eq:opt_robust_throughput}), it still remains convex as a minimum of convex functions. Therefore the outer problem minimizes a convex function over a convex set and is hence a convex problem.

We suggest to solve this problem by a gradient descent method. In order to do this, we need to compute the gradient of the inner optimization problem with respect to the argument $b$. Since the middle layer of problem \eqref{eq:opt_robust_throughput} is a discrete optimization problem, it does have an effect on the gradient only in that it determines a subset of $q$ entries of the capacity vector $b$ which are set to zero in the innermost LP. The derivative of the objective value with respect to the capacity vector can then be computed from the optimal simplex tableau obtained from the innermost optimization layer, as explained in Section \ref{sec:derivatives_of_LP_objective}.

\medskip

Alternatively, taking into account that the objective of the outer optimization problem is piece-wise affine, we may add a linear constraint each time the value and gradient is evaluated at a point. Let us describe this method in detail.

Let $f(x)$ be a piece-wise affine convex function. Suppose we want to minimize $f(x)$ over a polytope $P = \{ x \mid Ax \leq b \}$. Then we may proceed as follows.

\begin{algorithm}[H]
	\caption{Minimization of piece-wise affine function}
	\begin{algorithmic}[1]
		\Require Start point $x_0 \in P$. Constraint set $= \emptyset$
		\Ensure Optimal $x^{\star} \in P$,
		\For{$t \in \{0, \dotsc, \infty\}$}
		\State compute $f_t = f(x_t)$, $g_t = \nabla f(x_t)$
		\State add constraint $\phi \geq f_t + \langle g_t,x-x_t \rangle$ to constraint set
		\State minimize $\phi$ over $(\phi,x) \in \mathbb R \times P$ subject to constraints
		\State define $x_{t+1}$ as the minimizer
		\If{$x_{t+1} = x_t$}
		\State Break
		\EndIf
		\EndFor
		\State $x^* \gets x_t$
	\end{algorithmic}
\end{algorithm}

Each minimization problem is an LP, and the next LP differs from the previous one by an additional constraint. Hence the next LP can be solved by the dual simplex method when warm-started at the optimal basis of the previous LP, as described in Section \ref{sec:adding_constraints}. Note that the very first LP can be solved analytically, since we just minimize an affine function over a simplex.

\subsection{Robust latency optimization}

Latency optimization is more complicated and will be considered at a later stage. In this paper we only include numerical simulations.

\section{Numerical simulations}

\begin{figure}[h!]
    \centering
    \includegraphics[width=0.48\linewidth]{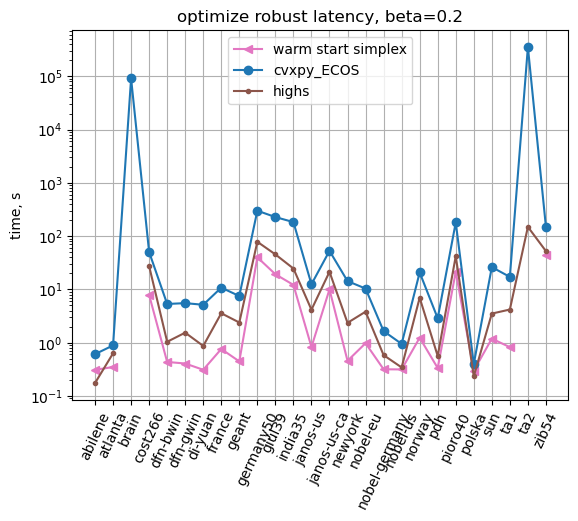}
    \includegraphics[width=0.48\linewidth]{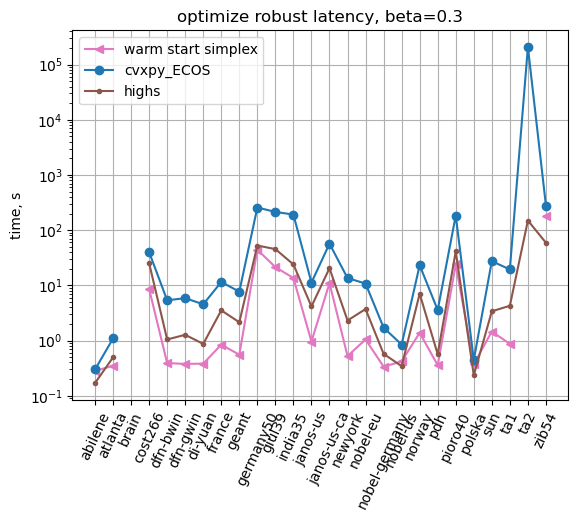} \\
    \includegraphics[width=0.48\linewidth]{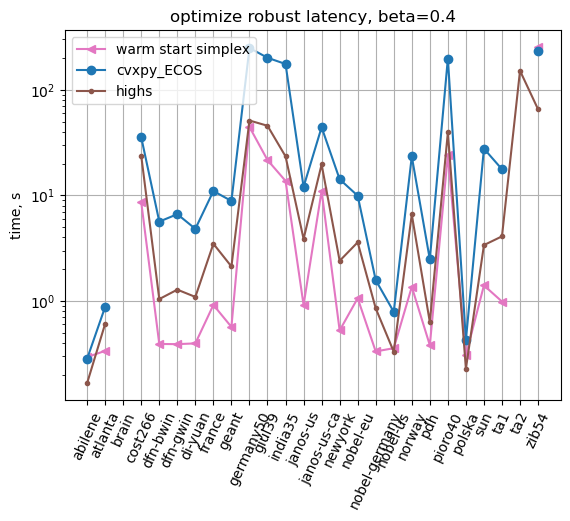}
    \includegraphics[width=0.48\linewidth]{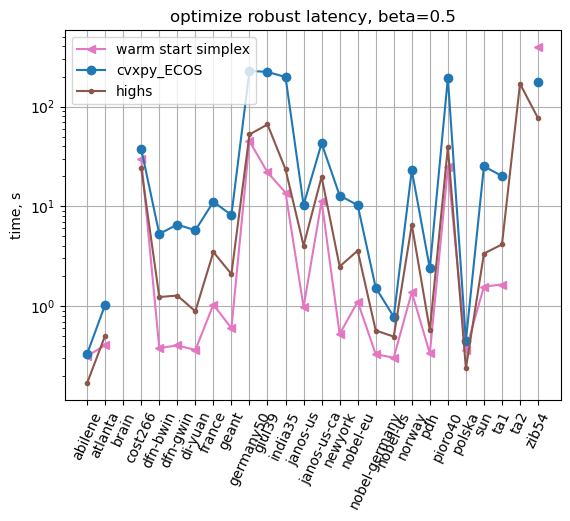} \\
    \includegraphics[width=0.48\linewidth]{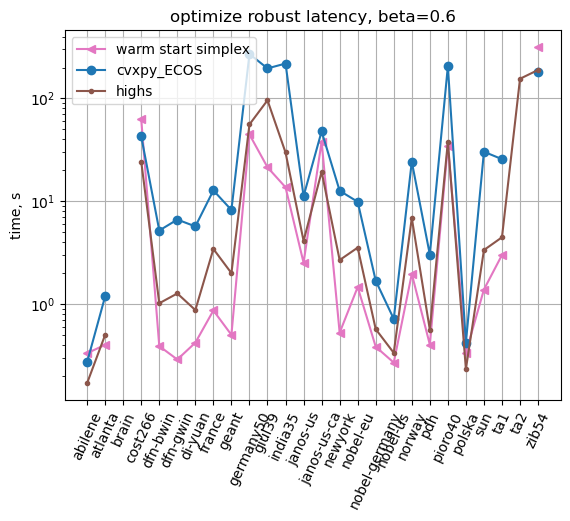}
    \includegraphics[width=0.48\linewidth]{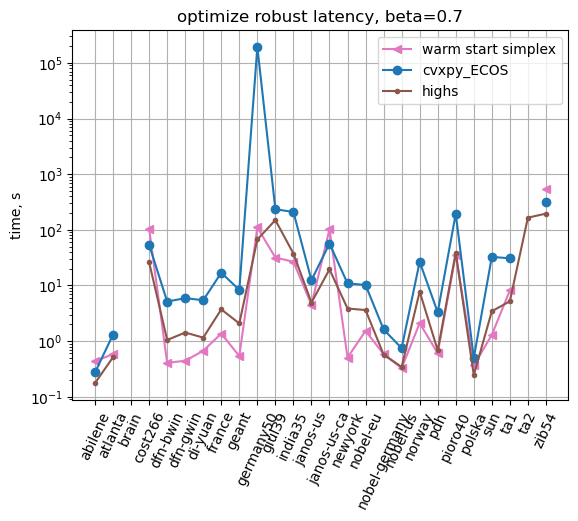}
    \caption{Optimization of robust latency}
    \label{fig:optimize_robust_latency}
\end{figure}

In this paper, two approaches to solution of robust network flow problem were studied: formulating a larger linear program via introducing new variables and solving a multi-layer optimization problem with linear programming at the lower level. In the latter case, a warm started simplex method was proposed. We compare the latter approach ("warm start simplex" in the figures) against solving the problem with additional variables. Problem solution methods include simplex method ("highs") and interior point algorithm ("cvxpy\_ECOS") both implemented in CVXPY library \cite{diamond2016cvxpy}. The comparison is done on telecom networks from SNDLib connection \cite{orlowski2010sndlib}.

\begin{figure}[H]
    \centering
    \includegraphics[width=0.8\linewidth]{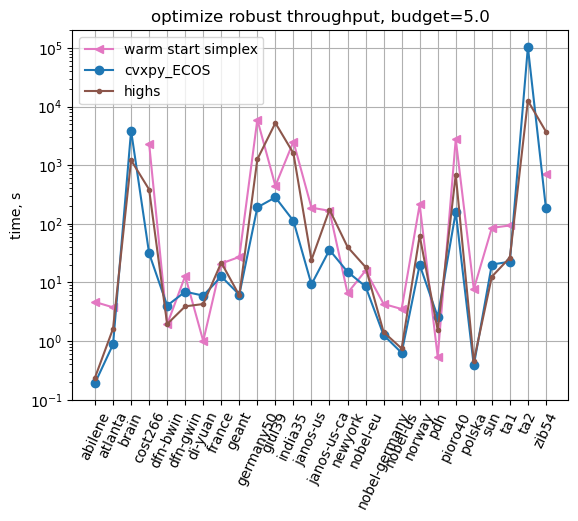}
    \caption{Optimization of robust throughput}
    \label{fig:enter-label}
\end{figure}

The plots show that warm started simplex method is competitive to its counterparts in all scenarios and even outperforms others almost by an order in several cases.

\section{Conclusions} \label{sec:conclusions}

In this work we considered the robust multi-commodity flow problem through a network described by a graph. The robustness is understood against failure, or deletion, of edges. Robustification of the network is accomplished by allocating a fixed budget of additional capacities to the edges of the network, in order to minimize the maximal performance degradation which can result from the failure of a constrained number of edges.

We considered two methods to solve the optimal robustification problem. The first relied on explicitly enumerating the values of the discrete variable in the problem formulation. It results in a large linear program. The second method leans on branch-and-bound methods known from integer linear programming and uses the dual simplex algorithm with warm start from the parent node.

Numerical experiments show that the second approach is more efficient. It can be potentially applied to other optimal network robustification problems, if the number of failed nodes remains small.

\bibliographystyle{plain}
\bibliography{robustThroughputPreprint}

\end{document}